\documentclass[11pt]{amsart}

\usepackage{epsfig}

\usepackage{lipsum}

\usepackage{graphicx}
\usepackage{amscd}
\usepackage{amsmath}
\usepackage{amsfonts}
\usepackage{mathrsfs}
\usepackage{amssymb}
\usepackage{verbatim}
\usepackage{color}
\usepackage{hyperref}

\newtheorem{theorem}{Theorem}[section]
\theoremstyle{plain}

\newtheorem{corollary}[theorem]{Corollary}

\newtheorem{defi}[theorem]{Definition}
\newtheorem{example}[theorem]{Example}

\newtheorem{lemma}[theorem]{Lemma}

\newtheorem{prop}[theorem]{Proposition}
\newtheorem{remark}[theorem]{Remark}

\numberwithin{equation}{section}

\def\B{{\mathcal B}}
 \def\Nset{\mathbb{N}}
 
 \def\Rset{\mathbb{R}}

\def\what{\widehat}

\def\Xx{{\mathfrak X}}
\def\FrA{{\mathfrak A}}

\def\FrR{{\mathfrak R}}

\def\Aa{{\mathbb A}}

\def\Ga{{\mathbb G}}
\def\One{{1\!\!1}}

\def\Lip{{\rm Lip}}

\def\Dh{\dim_{\rm H}}

\def\half{\frac{1}{2}}

\newcommand{\lam}{\lambda}

\newcommand{\om}{\omega}

\def\Om{\Omega}

\newcommand{\Gam}{\Gamma}
\newcommand{\sig}{\sigma}

\newcommand{\R}{{\mathbb R}}
\newcommand{\Q}{{\mathbb Q}}
\newcommand{\Z}{{\mathbb Z}}
\newcommand{\C}{{\mathbb C}}

\def\N{{\mathbb N}}

\newcommand{\E}{{\mathbb E}\,}
\newcommand{\Prob}{{\mathbb P}\,}
\def\P{\Prob}

\def\bq{{\bf q}}
\def\ba{{\bf a}}
\def\wt{\widetilde}
\def\A{{\mathcal A}}

\def\Mk{{\mathcal M}}

\def\Ik{{\mathcal I}}

\def\Pk{{\mathcal P}}

\def\Sf{{\sf S}}
\def\Tf{{\sf T}}
\def\T{{\mathbb T}}

\def\bx{{\mathbf x}}
\def\bt{{\mathbf t}}
\def\bz{{\mathbf z}}

\def\one{\vec{1}}
\def\be{\begin{equation}}
\def\ee{\end{equation}}

\newcommand{\eps}{{\varepsilon}}

\def\und{\underline}

\def\Cc{{\mathscr M}}
\def\Mc{{\mathscr M}}

\def\ve1{\vec{1}}
\def\Crs{{\mathfrak C}}

\def\Ak{{\mathcal A}}

\def\bn{{\bf n}}
\def\mfr{\mathfrak}

\begin{document}

\title{A note on spectral properties of random $S$-adic systems}

\author{Boris Solomyak \smallskip\\
\MakeLowercase{with an appendix by}  Pascal Hubert and Carlos Matheus }

\address{Boris Solomyak\\ Department of Mathematics,
Bar-Ilan University, Ramat-Gan, Israel}
\email{bsolom3@gmail.com}

\address{Pascal Hubert: Aix-Marseille Universit\'e, CNRS, Centrale Marseille, Institut de Math\'ematiques de Marseille, I2M - UMR 7373, 13453 Marseille, France.}
\email{pascal.hubert@univ-amu.fr}

\address{Carlos Matheus: Centre de Math\'ematiques Laurent Schwartz, CNRS (UMR 7640), \'Ecole Polytechnique, 91128 Palaiseau, France.}
\email{carlos.matheus@math.cnrs.fr}
\urladdr{http://carlos.matheus.perso.math.cnrs.fr}

\thanks{The research of Solomyak  was supported by the Israel Science Foundation grant \#1647/23}

\begin{abstract} 
The paper is concerned with random $S$-adic systems arising from an i.i.d.\ sequence of unimodular substitutions.
Using equidistribution results of Benoist and Quint, we show in Theorem 3.3 that, under some natural assumptions, if the Lyapunov exponent of the
spectral cocycle is strictly less than 1/2 of the Lyapunov exponent of the random walk on $SL(2,\R)$ driven by the sequence of substitution matrices, then almost surely
the spectrum of the $S$-adic $\Z$-action is singular with respect to any (fixed in advance) continuous measure. Finally, the appendix discusses the weak-mixing property for random $S$-adic systems associated to the family of substitutions introduced in Example 4.1. 
 \end{abstract}

\date{\today}


\maketitle

\begin{flushright}
{ \large \em Dedicated to the memory of Anatoly Moiseevich Vershik (1933--2024)}
\end{flushright}

\bigskip

\thispagestyle{empty}

\section{Introduction}

The paper is devoted to the spectral theory of $S$-adic dynamical systems. They are 
closely related to {\em Bratteli-Vershik} (BV) transformations, introduced by A. M. Vershik \cite{Ver1,Ver2} (who called them ``adic''), see also \cite{VerLiv}, as universal models of ergodic measure-preserving 
systems. 
A {\em substitution $\Z$-action} is defined as the shift on the space of sequences obtained by a repeated iteration of a single substitution on a 
finite alphabet. Ferenczi \cite{Feren}  introduced a generalization, in which a sequence of (possibly different) substitutions is applied at each step, in a predetermined order; he called such a 
system {\em $S$-adic}. Under some restrictions (most importantly, {\em recognizability}, see Definition~\ref{def-recog} below), every $S$-adic
system is measurably isomorphic to a ``natural'' BV-system, for any fully supported invariant measure, see \cite[Theorem 6.5]{BSTY}.

The literature on Bratteli-Vershik transformations, 
$S$-adic systems, and their spectral properties is vast; we mention a few papers, but do not attempt to provide an exhaustive survey. 
For the background on BV transformations as models for Cantor minimal systems, 
see  \cite{HPS,DHS} and the recent book \cite{DP}; for $S$-adic systems, see \cite{BD,BST19,BSTY}, as well as \cite[6.4]{DP}.
A detailed analysis of eigenvalues (both measurable and continuous) for Cantor minimal systems
is given in \cite{DFM19}; see also \cite{BDM10} and references therein for earlier work on this topic. Often BV systems and/or $S$-adic systems are studied with a motivation 
coming from a specific application: e.g., to interval exchange
transformations (IET's) and translation flows \cite{Bu13,Bu14,BuSo18,BuSo20,BuSo21,AFS}, or to multi-dimensional continued fraction algorithms and symbolic coding of toral translations
\cite{BST19,BST23}.
See also \cite{BoSe,BMY,AD,DMR} 
for the study of some special families of  $S$-adic systems. 
Quite often, individual systems are hard to analyze, and instead one tries to obtain results for ``almost every,'' in an appropriate sense, system in a particular class.
A famous example of this kind is a theorem of Avila and Forni \cite{AF} saying that almost every IET, that is not a rotation, is weakly mixing. (This was recently upgraded to 
``quantitative weak mixing with a polynomial rate'' for non-rotation class IET's in \cite{AFS}.)

In this note, following \cite{BuSo18,BuSo20,BuSo21}, we consider {\em random $S$-adic systems}. 
This means that a sequence of substitutions is chosen randomly from a given ``alphabet''; in general, the choice is driven by some ergodic stochastic process, but often it is i.i.d.
Recently, spectral properties of random $S$-adic systems
have been investigated
in several works; among them we should point out the paper \cite{BST23} by Berth\'e, Steiner, and Thuswaldner, where
the emphasis is on the generalized Pisot property and pure discrete spectrum. In contrast, in \cite{BuSo18,BuSo20,BuSo21} we were motivated by the study 
of translation flows on surfaces of higher genus and IET's, which are typically weakly mixing. 
Earlier, random 2-sided BV-systems were introduced by Bufetov \cite{Bu13,Bu14}, who used them to obtain limit theorems for translation flows on flat surfaces.
There is also work on the spectral properties of ``$S$-adic tiling systems''
\cite{Nagai,Tre}.

Our goal here is to extend the results of \cite{BuSo22}, which gave an effective (at least, in some cases) sufficient condition for singularity of the spectrum for substitution
$\Z$-actions, to the setting of random $S$-adic transformations. 
This condition is stated in terms of the top Lyapunov exponent of the spectral (``twisted'') cocycle, introduced in \cite{BuSo20} and extensively studied in recent years. 
The main new idea is an application of powerful 
equidistribution results, due to Benoist and Quint \cite{BQ}. As a bonus, we prove that, under our assumptions,
almost every random $S$-adic system is not just pure singular, but is disjoint from any (chosen in advance) weakly mixing measure-preserving $\Z$-action.
We illustrate this result on a class of examples, where the computations were done ``by hand''. Using numerical methods on a computer, one should be able to considerably expand the class of 
examples, for which this method is applicable.


\section{Preliminaries}

\subsection{Substitutions and $S$-adic systems}
For the background on  substitutions and associated dynamical systems see, e.g., \cite{Queff,Fogg}. In describing our set-up we closely follow \cite{BuSo20,BuSo21}.
Consider the
alphabet $\Ak=\{1,\ldots,d\}$, and denote by $\Ak^+$ the set of finite (non-empty) words with letters in $\Ak$. A {\em substitution}  is a map $\zeta:\,\A\to \A^+$, extended to  $\A^+$ and $\A^{\N}$ by concatenation. 
The {\em substitution matrix} is defined by 
\be \label{sub-mat}
\Sf_\zeta (i,j) = \mbox{number of symbols}\ i\ \mbox{in the word}\ \zeta(j).
\ee
Let $\FrA$ be the set of substitutions $\zeta$ on $\Ak$ with the property that all letters appear in the set of words $\{\zeta(a):\,a\in \Ak\}$ and there exists $a$ such that $|\zeta(a)|>1$.
Consider $\ba =(\zeta_n)_{n\ge 1}$, a 1-sided sequence of substitutions on $\Ak$, called a {\em directive sequence}. We denote
$$
\zeta^{[n]} := \zeta_1\circ \cdots\circ\zeta_n,\ \ n\ge 1.
$$
Recall that $\Sf_{\zeta_1\circ \zeta_2} = \Sf_{\zeta_1}\Sf_{\zeta_2}$.
We will sometimes write
$
\Sf_j:= \Sf_{\zeta_j}\ \ \mbox{and}\ \ \Sf^{[n]}:= \Sf_{\zeta^{[n]}}.
$
We will also consider subwords of the sequence $\ba$ and the corresponding substitutions obtained by composition. Denote
\be \label{notation1}
\Sf_\bq = \Sf_n\cdots \Sf_\ell\ \ \ \mbox{for}\ \ \bq = \zeta_{n}\ldots\zeta_{\ell}.
\ee

Given $\ba$, let $X_{\ba}\subset \Ak^\Z$ be the subspace of all two-sided sequences whose every subword appears as a subword of 
$\zeta^{[n]}(b)$ for some $b\in \Ak$ and $n\ge 1$. Let $T$ be the left shift on $\Ak^\Z$; then $(X_{\ba},T)$ is a topological $S$-adic dynamical system.
We refer to \cite{BD,BST19,BSTY} for the background on $S$-adic shifts.
The following will be assumed throughout the paper:

\medskip

{\bf (A1)} {\em There  is a word $\bq\in \FrA^+$ which appears in $\ba$ infinitely often, for which $\Sf_\bq$ has all entries strictly positive.}
 
 \medskip
 
Property (A1) implies unique ergodicity of the $S$-adic shift, see \cite[Theorems 5.2 and 5.7]{BD};  the claim on unique ergodicity goes back to Furstenberg \cite[(16.13)]{Furst}. 

A directive sequence $\ba = (\zeta_j)_{j\ge 1}$ and the associated $S$-adic system $(X_\ba,T)$ are said to be {\em primitive} if for every $n\ge 1$ there exists $N>n$ such that $\Sf_n\cdots \Sf_{N-1}$ has strictly positive entries. Primitivity implies minimality by a result of Durand, see 
\cite[Lemma 6.4.5]{DP}. It is clear that property (A1) implies primitivity.

We will want the $S$-adic system to be {\em aperiodic}, i.e., to have no periodic points under the shift action.
Aperiodicity is not always easy to check;  we quote some sufficient conditions which suit us well.

A substitution $\zeta$ on $\Ak$ is called {\em left (resp.\,right) proper} if all words $\zeta(a),\ a\in \Ak$, start (resp.\ end) with the same letter.
Another commonly used condition is {\em strong coincidence} \cite{AI01}. A substitution $\zeta$ is said to satisfy the strong coincidence condition 
if there exists $k\in \N$ and a letter $b\in \Ak$ such  that for every letter $a\in \Ak$ we have $\zeta^k(a) = W^a_{k}\, b\, S^a_{k}$, where
$W^a_{k}$ and $S^a_{k}$ are words, possibly empty, with the property that either all $W^a_{k},\ a\in \Ak$, or all $S^a_{k},\ a\in \Ak$,
have the same ``abelianization''.
It is clear that being left or right proper implies strong coincidence.

\begin{prop}[{\cite[Prop.\,2.1]{AD}, \cite{AD1}; see also \cite[Lemma 3.3]{BCDLPP21}}] \label{prop-aper}
Suppose that the directive sequence $\ba = (\zeta_j)_{j\ge 1}$ is primitive, $\det(\Sf_{\zeta_j})\ne 0$ for all $j$, and such that either all $\zeta_j$ are 
left proper, or all $\zeta_j$ are right proper, or
there is a word $\bq$ appearing in $\ba$ infinitely often, with $\zeta_\bq$ satisfying the strong coincidence
condition.
Then the $S$-adic system $(X_\ba,T)$ is aperiodic.
\end{prop}

\begin{remark}[{\cite{AD1}}]
{\em 
It should be noted that the condition of being proper (or with a strong coincidence) 
is missing in the assumptions of \cite[Prop.\,2.1]{AD}. Without this condition the claim is false: e.g.,
take $\zeta_j = \zeta$ for all $j$ to be $\zeta(0) = 010,\ \zeta(1) = 101$.
}
\end{remark}

Another important condition we need is {\em recognizability}.
For a sequence of substitutions this notion was introduced 
in \cite{BSTY}; it generalizes {\em bilateral recognizability} of  Moss\'e \cite{Mosse} for a single substitution (see Sections 5.5 and 5.6 in \cite{Queff}). By the 
definition of the space $X_{\ba}$, for every $n\ge 1$, every $x\in X_{\ba}$ has a representation of the form
\be \label{recog}
x = T^k\bigl(\zeta^{[n]}(x')\bigr),\ \ \mbox{where}\ \ x'\in X_{\sig^n \ba},\ \ 0\le k < |\zeta^{[n]}(x_0)|,
\ee
where $\sig$ is the left shift on $\FrA^\N$.  The action of a substitution $\zeta$ is extended to two-sided sequences in $\Ak^\Z$ by
\be \label{extend}
\zeta(\ldots a_{-1}.a_0 a_1\ldots) = \ldots \zeta(a_{-1}).\zeta(a_0)\zeta(a_1)\ldots
\ee

\begin{defi} \label{def-recog}
A sequence of substitutions $\ba = (\zeta_j)_{j\ge 1}$ is said to be {\em recognizable} if the representation (\ref{recog}) is unique for all $n\ge 1$.
\end{defi}

The following is a special case of \cite[Theorem 4.6]{BSTY} that we need.

\begin{theorem}[{\cite{BSTY}}] \label{th-recog0}
Let $\ba= (\zeta_j)_{j\ge 1} \in \FrA^\N$ be such that $\det(\Sf_{\zeta_j})\ne 0$ for every substitution matrix and $X_{\ba}$ is aperiodic. Then $\ba$ is recognizable.
\end{theorem}


\smallskip

\subsubsection{Random $S$-adic systems} \label{sec:rand}
Let $\Om$ be a shift-invariant subspace of $\FrA^\N$, with an invariant measure $\P$. We suppose that the following properties hold:

\smallskip

{\bf (C1)} the system $(\Om,\sig,\P)$ is ergodic;

\smallskip

{\bf (C2)} the function $\ba \mapsto \log(1 + \|\Sf_{\zeta_1}\|)$ is integrable;

\smallskip

{\bf (C3)} there is a word $\bq$ admissible for sequences in $\Om$, such that all entries of the matrix $\Sf_\bq$ are positive and $\P([\bq])>0$.

\smallskip

{\bf (C4)} the set $\FrR$ of substitutions from $\FrA$ which appear in $\Om$ is {\em finite}, and
\be \label{assume1}
\forall\,\zeta\in \FrR,\ \ \det(\Sf_\zeta) \ne 0.
\ee

Observe that {\bf (C1)} and {\bf (C3)} imply {\bf (A1)} for $\P$-a.e.\ $\ba$, hence unique ergodicity and primitivity of the random $S$-adic system
$(X_\ba,T)$ holds $\P$-almost surely. Moreover, {\bf (A1)}, together
with {\bf (C4)}, yield recognizability for $\P$-a.e.\ $\ba$, by Theorem~\ref{th-recog0}.

\smallskip

\subsection{Spectral cocycle}
We recall the definition of the spectral cocycle introduced in \cite{BuSo20}.
Let $\zeta$ be a substitution on $d$ symbols, with a substitution matrix $\Sf_\zeta$, having a non-zero determinant.
Then the map
$\bt\mapsto \Sf_\zeta^\Tf \,\bt\ (\mbox{mod}\ \Z^d),\ \bt \in \T^d = \R^d/\Z^d$, is a toral endomorphism, which preserves the Haar measure $m_d$ on $\T^d$.
Here and below the symbol $\Sf^\Tf$ denotes the matrix transpose of $\Sf$.

Let us write the substitution map on symbols explicitly, as follows:
$$\zeta(b) = u_1^{b}\ldots u_{|\zeta(b)|}^{b},\ \ b\in \Ak.$$
Consider the matrix-valued function $\Cc_\zeta:  \R^d\to M_d(\C)$  (the space of complex $d\times d$ matrices) defined by 
\be \label{coc0}
\Cc_\zeta(\bt) = [\Cc_\zeta(t_1\ldots,t_d)]_{(b,c)} := \Bigl( \sum_{j\le |\zeta(b)|,\ u_j^{b} = c} \exp\bigl(-2\pi i \sum_{k=1}^{j-1} t_{u_k^{b}}\bigr)\Bigr)_{(b,c)\in \A^2},\ \ \ \bt\in \R^d.
\ee
Observe that  $\Cc_\zeta$ is $\Z^d$-periodic, so we obtain a continuous matrix-function on the torus, which we denote, by a slight abuse of notation, by the same letter:
$\Cc_\zeta: \T^d\to M_d(\C)$.
Note that $\Cc_\zeta(\bt)$ is a matrix-function whose entries are trigonometric polynomials in $d$ variables, with the following properties: 

\smallskip

(i) $\Cc_\zeta(0) = \Sf^{\sf T}_\zeta$, and every entry  of the matrix $\Cc_\zeta(\bt)$ is less or equal to the corresponding entry of $\Sf^\Tf_\zeta$ in absolute value;

\smallskip

(ii) all the coefficients of the trigonometric polynomials are  0's and 1's, and in every row, any given monomial appears at most once;

\smallskip

(iii)  the substitution is uniquely determined by 
$\Cc_\zeta$. 

\smallskip

(iv) the cocycle property: for any substitutions $\zeta_1$ and $\zeta_2$ on the same alphabet we have
$$
\Cc_{\zeta_1\circ \zeta_2}(\bt) = \Cc_{\zeta_2}(\Sf^{\sf T}_{\zeta_1}\bt)\Cc_{\zeta_1}(\bt),
$$
which is verified by a  direct computation.

\begin{defi} \label{def-cocycle2}
Let $(\Om,\sig,\P)$ be a measure-preserving shift on $\FrA^\N$ satisfying the properties {\bf (C1)--(C4)}. 
Consider the skew product transformation $\Ga:\, \Om \times \T^d \to \Om \times \T^d$ defined by
\be \label{skew1}
\Ga(\ba, \bt) = \bigl(\sig \ba, \Sf^\Tf_{\zeta_1} \bt \, ({\rm mod}\ \Z^d)\bigr),\ \ \mbox{where}\ \ \ba = (\zeta_n)_{n\ge 1}\ \ \mbox{and}\ \ \bt \in \T^d = \R^d/\Z^d.
\ee
Let $\Cc(\ba,\bt) = \Cc_{\zeta_1}(\bt)$ be defined by \eqref{coc0},
where $\ba = (\zeta_n)_{n\ge 1}$.
Then
\begin{equation}\label{def-cc}
\Cc_{_\Om}((\ba,\bt),n):= \Cc(\Ga^{n-1}(\ba,\bt))\cdot \ldots \cdot \Cc(\ba,\bt)
\end{equation}
is a complex matrix cocycle over the skew product system $(\Om\times\T^d, \P\times m_d, \Ga)$.
\end{defi}

We also need the ``untwisted'' cocycle (the analog of the Rauzy-Veech cocycle from the theory of IET's), which is simply $\Cc_{_\Om}((\ba,0),n)$:
 \be \label{Rauzy-Veech1}
 \Aa(\ba):= \Sf_{\zeta_1}^{\sf T};\ \ \Aa(\ba,n) := \Aa(\sig^{n-1}\ba)\cdot \ldots \cdot \Aa(\ba),
 \ee
 where $\ba = (\zeta_j)_{j=1}^\infty\in \Om$.
 This is a $GL(d,\R)$-cocycle over the ergodic system $(\Om,\sig,\P)$. Property {\bf (C2)} implies that
 the
 corresponding Lyapunov exponent exists a.e.:
 \be \label{RV2}
 \lam:= \lim_{n\to \infty} \frac{1}{n} \log\|\Aa(\ba,n)\|\ \ \mbox{for $\P$-a.e.\ $\ba\in \Om$}.
 \ee
In fact, $\lam>0$, which can be easily deduced from assumption {\bf (C3)}.

\subsection{Lyapunov exponents of the spectral cocycle}

Consider the pointwise upper Lyapunov exponent (which is defined everywhere and is independent of the matrix norm):
\begin{equation} \label{Lyap1}
{\chi}_{\ba,\bt}^+:= \limsup_{n\to \infty} \frac{1}{n} \log \|\Cc_{_\Om}((\ba,\bt),n)\|,\ \ \bt\in \T^d.
\end{equation}
Note  that the
property (i) of $\Cc_\zeta$ above implies
\be \label{Lyap110}
\chi^+_{\ba,\bt} \le \lam\ \ \mbox{for all}\ \ba\in \Om,\ \bt\in \T^d.
\ee

Now assume that $(\Om\times\T^d, \P \times m_d, \Ga)$ is ergodic. Then by theorems of Furstenberg-Kesten \cite{FK} and Kingman \cite{Kingman} there is a ``global'' Lyapunov exponent
\begin{eqnarray} \nonumber
\chi(\Cc_{_\Om}) & = & \lim_{n\to \infty} \frac{1}{n} \log \|\Cc_{_\Om}((\ba,\bt),n)\|,\ \ \mbox{for $(\P\times m_d)$-a.e.}\ (\ba,\bt)\\[1.2ex]
& = &  \inf_k \frac{1}{k} \int_{\Om}\int_{\T^d} \log\|\Cc_{_\Om}((\ba,\bt),k)\|\,dm_d(\bt)\,d\P(\ba). \label{Lyap12}
\end{eqnarray}

The value of the Lyapunov exponent does not depend on the norm; often it will be convenient to use the Frobenius norm of a matrix, defined by
$${\|(a_{ij})_{i,j}\|}^2_{\rm F} = \sum_{i,j} |a_{ij}|^2.$$ 

\begin{lemma} \label{lem:lower}
For any $k\ge 1$, the function $(\ba,\bt) \mapsto \log\|\Cc(\ba,\bt),k)\|$ is integrable, and
$$
\int_{\Om}\int_{\T^d} \log\|\Cc_{_\Om}((\ba,\bt),k)\|\,dm_d(\bt)\,d\P(\ba)\ge 0.
$$
Thus, $\chi(\Cc_{_\Om})\ge 0$.
\end{lemma}

\begin{proof} This is an immediate consequence of \cite[Lemma 2.3]{BuSo22}, but we recall the argument for the reader's convenience. By Definition~\ref{coc0},  we obtain that for any substitution $\zeta$ on $\Ak$, the function ${\|\Cc_\zeta(\bt)\|}_{\rm F}^2$ is a sum of the squares of absolute values of
polynomials in $d$ variables $z_j = e^{-2\pi i t_j}$. This expression can be rewritten as a (positive) polynomial in the variables $z_j^{\pm 1}$. Multiplying it by
 $z_j^{\ell_j}$ for some $\ell_j\ge 0$ to get rid of the negative powers, we obtain that 
 $$
 \log {\|\Cc_\zeta(\bt)\|}_{\rm F}^2 = \log|P_\zeta(z_1,\ldots,z_d)|
 $$
 for some polynomial $P_\zeta$ with integer coefficients, It follows that 
 $$
 \int_{\T^d} \log {\|\Cc_\zeta(\bt)\|}_{\rm F}^2\,dm_d(\bt) = \mfr{m}(P_\zeta)\ge 0,
$$
 where $\mfr{m}(P_\zeta)$ is the logarithmic Mahler measure of a polynomial, see \cite{Boyd}. Since
 $$
 \Cc_{_\Om}((\ba,\bt),k) = \Cc_{\zeta^{[k]}}(\bt),
 $$
 where $\zeta^{[k]}$ corresponds to the first $k$ letters of $\ba$, the claim follows.
\end{proof}

It is immediate from \eqref{Lyap110} that $\chi(\Cc_\Om) \le  \lam$; however, this is not sharp.
The next lemma follows from \cite[Corollary 4.5]{BuSo20}, but we give a direct proof here, inspired by the argument in  \cite[Theorem 3.29]{BG2M}.
 
\begin{lemma} \label{lem:upper} Suppose that {\bf (C1)--(C4)} are satisfied and the skew product $(\Om\times\T^d, \P \times m_d, \Ga)$ is ergodic. Then
$\chi(\Cc_\Om) \le  \half\lam$.
\end{lemma}

\begin{proof}
We have by Jensen's inequality:
\begin{eqnarray} \nonumber
\Ik_k:=\frac{1}{k} \int\limits_{\Om}\int\limits_{\T^d} \log{\|\Cc_{_\Om}((\ba,\bt),k)\|}_{\rm F}\,dm_d(\bt)\,d\P(\ba) & = & 
\frac{1}{2k} \int\limits_{\Om}\int\limits_{\T^d} \log{\|\Cc_{_\Om}((\ba,\bt),k)\|}^2_{\rm F}\,dm_d(\bt)\,d\P(\ba) \\
& \le & \frac{1}{2k} \log \int\limits_{\Om}\int\limits_{\T^d} {\|\Cc_{_\Om}((\ba,\bt),k)\|}^2_{\rm F}\,dm_d(\bt)\,d\P(\ba). \label{cri}
\end{eqnarray}
By the definition of Frobenius norm and Parseval's formula,
\begin{eqnarray*}
\int_{\T^d} {\|\Cc_{_\Om}((\ba,\bt),k)\|}^2_{\rm F}\,dm_d(\bt) & = & \sum_{(b,c)\in \Ak^2} \int_{\T^d} |(\Cc_{_\Om}((\ba,\bt),k))_{(b,c)}|^2\,dm_d(\bt) \\
& = & \sum_{(b,c)\in \Ak^2} \sum_{\bn\in \Z^d} |(\what\Cc_{_\Om}((\ba,\cdot),k))_{(b,c)}(\bn)|^2 \\
& = & \sum_{(b,c)\in \Ak^2} \sum_{\bn\in \Z^d} |(\what\Cc_{_\Om}((\ba,\cdot),k))_{(b,c)}(\bn)|,
\end{eqnarray*}
since all the coefficients are 0's and 1's. Moreover, the sum of all the coefficients is exactly the sum of all the entries of the corresponding substitution matrix. Thus we obtain
$$
\int_{\T^d} {\|\Cc_{_\Om}((\ba,\bt),k)\|}^2_{\rm F}\,dm_d(\bt)  = \sum_{b\in \Ak} |\zeta_\ba^{[k]}(b)|= \sum_{(b,c)\in \Ak^2} \Sf_{\zeta_\ba^{[k]}}(b,c) = \|A(\ba,k)\|_{1,1},
$$
where
$$
{\bigl\|(a_{ij})_{i,j}\bigr\|}_{1,1} = \sum_{i,j} |a_{ij}|.
$$
By Egorov's Theorem and \eqref{RV2}, for any $\eps>0$ there exists $E_\eps\subset \Om$, with $\P(E_\eps) < \eps$, and $N=N(\eps)\in \N$, such that for all $k\ge N$ holds
$$
{\bigl\|\Aa(\ba,k)\bigr\|}_{1,1} \le e^{k(\lam+\eps)}\ \ \mbox{for all}\ \ba\in \Om\setminus E_\eps.
$$
On the other hand, for all $\ba\in \Om$ and all $\bt\in \T^d$, we have
$$
{\|\Cc(\ba,\bt)\|}_{\rm F} \le C_1:= \max_{\zeta\in \FrR}{\|\Sf_\zeta\|}_{\rm F} \implies  {\|\Cc_{_\Om}((\ba,\bt),k)\|}_{\rm F}\le C_1^k,
$$
since the Frobenius norm is sub-multiplicative.
Thus, splitting the integral over $\Om$ in \eqref{cri} into the sum of integrals over $\E_\eps$ and over its complement, we obtain
$$
\Ik_k \le \eps \log C_1 + \textstyle{\half} (\lam+\eps),\ \ k\ge N(\eps).
$$
Letting $\eps\to 0$ yields the desired inequality $ \chi(\Cc_{_\Om})\le \half \lam$, in view of \eqref{Lyap12}.
\end{proof}

\subsection{Spectral measures, cylindrical functions} \label{spec1}
Now we return to the random $S$-adic system $(X_\ba,T)$ under the assumptions {\bf (C1)--(C4)}.
As we have seen, they imply unique ergodicity of $(X_\ba,T)$ for $\P$-a.e.\ $\ba$. Denote the unique invariant Borel probability measure by $\mu_\ba$.
We are interested in spectral properties of the measure-preserving system $(X_{\ba},T, \mu_\ba)$. 
Here and below, by {\em spectral properties of a measure-preserving $\Z$-action} we always mean the properties of the corresponding unitary operator ---
the Koopman operator, and similarly, for the $\R$-action --- the properties of the associated unitary group.

Recall that, given a test function $f\in L^2(X_\ba,\mu_\ba)$,
the {\em spectral measure} $\sig_f$ is a finite positive measure on the unit circle $\T\cong \R/\Z$, determined by its Fourier coefficients:
$$
\what\sig_f(-k) = \int_0^1 e^{2\pi i k\om}\,d\sig_f(\om) = \langle U_T^k f\,,\,f\rangle,\ \ \ k\in \Z,
$$
where $U_T: f\mapsto f\circ T$ is the Koopman operator on $L^2(X_\ba,\mu_\ba)$ and $\langle \cdot,\cdot\rangle$ denotes the inner product in $L^2 $.
It is well-known that, given a countable complete (i.e., having
a dense linear span) set $\{f_n\}_{n=1}^\infty$ of unit vectors in $L^2$, the maximal spectral type of $U_T$, equivalently, of the measure-preserving system, is
given by
\be \label{sigmax}
[\sig_{\max}]=\textstyle{\bigl[\sum_{n=1}^\infty 2^{-n} \sig_{f_n}\bigr]}.
\ee
The symbol $[\nu]$ denotes the {\em type} of the measure $\nu$, that is, the equivalence class of $\nu$ under the relation of mutual absolute continuity.

The most basic functions in the symbolic space $X_\ba$ are the functions depending only on the symbol $x_0$. We call them {\em cylindrical functions of
level} 0. Of course, their span is only $d$-dimensional, where $d$ is the number of symbols, so in order to ``capture'' the maximal spectral type we need
functions depending on an arbitrary finite number of symbols. In order to do this efficiently, we make use of renormalization.

Recall that the action of substitutions $\zeta\in \FrR$ is extended to $\Ak^\Z$ in such a way that $\zeta(x_0)$ starts from the $0$-th position, see
\eqref{extend}. Consider the subsets $\zeta^{[n]}[a]$
of $X_\ba$, where $a\in \Ak$ and $[a]$ is the cylinder set of sequences having $x_0=a$. We emphasize that $\zeta^{[n]}[a]$ may be a proper subset of $[\zeta^{[n]}(a)]$.
Recognizability implies that 
$$
\Pk_n = \left\{T^i(\zeta^{[n]}[a]): a\in \Ak,\ 0 \le i < |\zeta^{[n]}(a)|\right\},\ \ n\ge 1,
$$
is a sequence of Kakutani-Rokhlin partitions, which generates the Borel $\sig$-algebra on $X_\ba$, see \cite[Lemma 6.3]{BSTY}. Functions that are 
measurable with respect to the partition $\Pk_\ell$, but not $\Pk_{\ell-1}$, are called {\em cylindrical functions of level} $\ell$. Thus a sequence $f_n$
of cylindrical functions of levels $\ell_n\to \infty$, having the property 
that their linear span contains all cylindrical functions of all levels, generates the maximal spectral type
of the $S$-adic system $(X_\ba,T, \mu_\ba)$, as in \eqref{sigmax}, for $\P$-a.e.\ $\ba$.

\subsection{Suspension flows and their spectrum} \label{sec:spec}
As in \cite{BuSo22}, we
also need to consider {\em suspension flows} (or {\em special flows}, in another terminology), over $S$-adic systems, for which the roof function
$\phi$ is piecewise-constant and depends only on the symbol $x_0$ of $x\in X_\ba$:
\be \label{roof}
\phi(x) = s_{x_0}\ge 0.
\ee
Consider the partition of $X_\ba$  into cylinder sets according to the value of $x_0$: $X_{\ba} = \bigsqcup_{a\in \Ak} [a] $.
Denote
by $(\Xx_\ba^{\vec{s}}, h_\tau, \wt{\mu}_\ba)$ the suspension flow over $(X_{\ba},T, \mu_\ba)$,  corresponding to the piecewise-constant roof function 
\eqref{roof}, where $\vec{s}=(s_a)_{a\in \Ak}\in \R^d_+$. 
The suspension flow is defined on the following union:
$$
\Xx_\ba^{\vec{s}} = \bigcup_{a\in \Ak} [a]\times [0,s_a],
$$
with the invariant measure $\wt\mu_\ba$ equal to the normalized sum of $\mu_\ba|_{[a]} \times m|_{[0,s_a]}$. Here $m$ denotes the Lebesgue measure
on $\R$.
We define
 a {\em Lip-cylindrical function of level 0} by the formula: 
\be \label{fcyl2}
F(x,\tau)=\sum_{a\in \Ak} \One_{[a]}(x) \cdot \psi_a(\tau),\ \ \mbox{with}\ \ \psi_a\in \Lip[0,s_a],
\ee
where $\Lip$ is the space of Lipschitz functions.
Further, $F$ is a {\em Lip-cylindrical function of level $\ell\ge 1$} if
\be \label{fcyl3}
F(x,\tau)=\sum_{a\in \Ak} \One_{\zeta^{[\ell]}[a]}(x) \cdot \psi^{(\ell)}_a(\tau),\ \ \mbox{with}\ \ \psi^{(\ell)}_a\in \Lip[0,s^{(\ell)}_a],
\ee
where
$$
\vec{s}^{\,(\ell)}= (s^{(\ell)}_a)_{a\in \Ak}:= \Sf_{\zeta^{[\ell]}}^{\sf T} \vec{s}.
$$
The union of spaces of Lip-cylindrical functions of level $\ell$ over $\ell\in \N$, is dense in $L^2(\Xx_\ba^{\vec{s}},\wt \mu_\ba)$, 
see \cite[Lemma 7.1]{BuSo20}, and hence the maximal spectral type
of $(\Xx_\ba^{\vec{s}},h_\tau,\wt \mu_\ba)$ can be recovered from the spectral measures $\sig_F$, 
where $F$ runs over Lip-cylindrical functions of an arbitrary level.

Recall that, given a probability measure-preserving flow $h_\tau$ on a space $\Xx$ and a test function $F\in L^2(\Xx)$, the spectral measure $\sig_F$ is a
finite positive Borel measure on $\R$ determined by the formula
$$
\what \sig_F(-\tau) = \int_\R e^{2\pi i \om \tau}\,d\sig_F(\om) = \langle F\circ h_\tau, F\rangle,\ \ \ \tau\in \R.
$$
It is well-known that spectral properties of a (general, probability-preserving) $\Z$-action on $(X,\mu)$
are closely related to those of the suspension $\R$-action $(\wt X,\wt\mu)$,
corresponding to the constant-1 roof function, in particular, when we have $\vec s = \one$ in the case of an $S$-adic system.
In fact, 
for  $f\in L^2(X,\mu)$ consider the
function $F\in L^2(\wt{X},\wt{\mu})$  defined by
$F(x,\tau)=f(x)$. Then
the following relation holds between the spectral measures $\sigma_F$  on $\R$ and $\sigma_f$ on $\T$:
\be \label{eq-spec}
d\sig_F(\om)=\left(\frac{\sin(\pi \om)}{\pi \om}\right)^2 \cdot d\sig_f(e^{2\pi i \om}),\ \ \om\in \R.
\ee
See \cite[Prop.\,1.1]{DL} for a general treatment and \cite[Lem.\,5.6]{BerSol} for our specific case. 

\medskip

We next state a result from \cite{BuSo20}, from which we quote only the relevant part. It provides a lower bound for the pointwise local dimension of 
a spectral measure, defined  by
$$
\und{d}(\sig_F,\om):= \liminf_{r\to 0+} \frac{\log \sig_F(B(\om,r))}{\log r}.
$$

\begin{theorem}[see {\cite[Theorems 4.3 and 4.6 ]{BuSo20}}] \label{th-BuSo}
Let $(\Om,\sig,\P)$ be a measure-preserving shift on $\FrA^\N$ satisfying the properties {\bf (C1)--(C4)}, such that 
for $\P$-a.e.\ $\ba\in \Om$, the $S$-adic system $(X_\ba,T)$ is aperiodic, and hence it is 
 recognizable and uniquely ergodic for $\P$-a.e.\ $\ba\in \Om$.
Denote by $\mu_\ba$ the unique invariant
probability measure.
Fix any $\vec s\in \R^d_+$ and consider the suspension flow $(\Xx_\ba^{\vec s}, h_\tau, \wt \mu_\ba)$ over $(X_\ba,T,\mu_\ba)$, 
with the piecewise constant roof function determined by $\vec s\in \R^d_+$. Then for $\P$-a.e.\ $\ba\in \Om$ the following holds:

Let $F$ be a nonzero Lip-cylindrical function on $\Xx_\ba^{\vec s}$ (of any level) and let
$\sig_F$ be the corresponding spectral measure for the flow. Then the lower local dimension of $\sig_F$ satisfies for all $\om\in \R$:
\be \label{dim-est}
\und{d}(\sig_F,\om) \ge 2 \min\Bigl\{1, 1- \frac{\chi^+_{\ba,\om\vec s}}{\lam}\Bigr\},
\ee
where $\lam$ is from \eqref{RV2} and $\chi^+_{\ba,\om\vec s}$ is the local upper Lyapunov exponent of the spectral cocycle, defined in
\eqref{Lyap1}, with $\bt =\om\vec s$ mod $\Z^d$.
\end{theorem}

\begin{corollary}[see {\cite[Corollary 4.5]{BuSo20}}] \label{cor-BuSo}
Under the assumptions of the last theorem, we have for any $\vec s\in \R^d_+$, for $\P$-a.e.\ $\ba\in \Om$, and for any nonzero Lip-cylindrical function $F$:
\be \label{Lyapu-est}
\chi^+_{\ba,\om\vec s} \ge \textstyle{\half}\lam\ \ \mbox{for}\ \sig_F\mbox{-a.e.}\ \om\in \R.
\ee
\end{corollary}

We recall the derivation of the corollary. The upper Hausdorff dimension of the measure $\sig_F$ is defined by 
$$
\Dh^*(\sig_F) = \inf \{\Dh(E):\ E\subset \R \ \mbox{is a Borel set with}\ \sig_F(\R\setminus E) = 0\}. 
$$
Obviously, $\Dh^*(\sig_F)\le 1$. Combining \eqref{dim-est} with the well-known formula
$$
\Dh^*(\sig_F) =
\inf\{s:\ \und{d}(\sig_F,\om) \le s\ \ \mbox{for $\sig_F$-a.e.}\ \ \om\},
$$
see, e.g., \cite[(10.12)]{Falconer3}, yields the claim. \qed

\subsection{Equidistribution results}

The next theorem is a special case of a result of Benoist and Quint \cite{BQ}. 

\begin{theorem}[Benoist and Quint] \label{th-equi1}
Let $\Gam$ be a semigroup in $SL(d,\R)$ which

\smallskip

{\bf (A)} acts strongly irreducibly on $\R^d$, that is, no finite union of proper vector subspaces of $\R^d$ is $\Gam$-invariant;

\smallskip

{\bf (B)} contains a proximal element, i.e., a matrix with a dominant eigenvalue.

\smallskip

\noindent Let $\nu$ be a finitely supported probability measure on $\Gam$, whose support generates $\Gam$.
Consider the random walk $g_n \ldots g_1$, where $g_i$ are chosen i.i.d., according to the measure $\nu$ on $SL(d,\R)$. Then for every
$x_0\in \T^d\setminus \Q^d/\Z^d$, the sequence
$$
(g_n \ldots g_1\cdot x_0)_{n\in \N}
$$
is almost surely (i.e., for a.e.\ realization of the random walk) equidistributed on the torus $\T^d$. 
\end{theorem}

We also need the following proposition, which we state in a special case.

\begin{prop}[{\cite[Prop.\,5.1]{SiWe}}] \label{prop-equi1}
Let $X$ be a compact space, $G$ a locally compact 2nd countable group acting continuously on $X$. Suppose that $m$ is a $G$-invariant and ergodic measure  on $X$, and $\mu$ a probability measure with finite support $E$ on $G$. Let $B = E^\N$, $\beta = \mu^{\N}$, and denote by $T$ the left shift on $B$. Fix $x_0\in X$ and suppose that for $\beta$-a.e.\ $b\in B$, the random path
$(g_{b_n}\ldots g_{b_1}\cdot x_0)_{n\in \N}$ is equidistributed with respect to the measure $m$ on $X$. Then for $\beta$-a.e.\ $b\in B$, the sequence
$$
(g_{b_n}\ldots g_{b_1}\cdot x_0, T^n b)_{n\in \N}
$$
is equidistributed with respect to the measure $m\times \beta$ on $X\times B$.
\end{prop}

Combining Theorem~\ref{th-equi1} with Proposition~\ref{prop-equi1}, we obtain that if $\nu$ is a probability measure on $\Gam$ with a finite support $E$, which acts on the Euclidean space $\R^d$ totally irreducibly and contains a proximal element, then given any $x_0\in \T^d\setminus \Q^d/\Z^d$, for 
almost every $b\in B$, the sequence $(g_{b_n}\ldots g_{b_1}\cdot x_0, T^n b)_{n\in \N}$ is equidistributed with respect to $m_d\times \beta$.


\section{A new result: application of equidistribution} 

\begin{theorem}\label{th:main}
For each $j\in \{1,\ldots,\ell\}$, let $\zeta_j$ be a
substitution on $d$ letters. Consider the infinite product space
 $\Om=\{\zeta_1,\ldots,\zeta_\ell\}^\N$, with $\ell\ge 2$, equipped with a fully supported
Bernoulli measure $\P = \mu^{\N}$.
Let $\Sf_j$ be the substitution matrix of $\zeta_j$, and let $\Gam$ be the semigroup generated by $\Sf_j^\Tf$, $j\le \ell$.
We suppose that the following properties hold:
\begin{enumerate}
\item[{\bf (B1)}] for each $j\le \ell$, $\det(\Sf_j)=1$, and for $\P$-a.e.\ $\ba\in \Om$ the space $X_\ba$ is aperiodic (e.g., all $\zeta_j$ are left or
right proper);
\item[{\bf (B2)}] the semigroup $\Gam<SL(d,\R)$ acts strongly irreducibly on $\R^d$;
\item[{\bf (B3)}] the semigroup $\Gam$  contains a matrix with strictly positive entries.
\end{enumerate}
Let $\lam$ be the top Lyapunov exponent of the cocycle induced by $\Sf^\Tf$, as in \eqref{RV2}. We have $\lam>0$.
Consider the skew product transformation, as in \eqref{skew1}: $\Ga:\, \Om\times \T^d \to \Om\times \T^d$, defined by
$$
\Ga(\ba, \bt) = \bigl(\sig \ba, \Sf^\Tf_{\zeta_1} \bt \, ({\rm mod}\ \Z^d)\bigr),\ \ \mbox{where}\ \ \ba = (\zeta_n)_{n\ge 1}\ \ \mbox{and}\ \ \bt \in \T^d = \R^d/\Z^d.
$$
Then

{\bf (i)} The properties {\bf (C1)--(C4)} are satisfied, and for $\P$-a.e.\ $\ba\in \Om$ the corresponding $S$-adic measure-preserving
system $(X_\ba,T,\mu_\ba)$ is aperiodic, recognizable, and uniquely ergodic.

{\bf (ii)}
The skew product $\Ga$ is ergodic, hence the top Lyapunov exponent $\chi(\Cc_{_\Om})$ exists, see \eqref{Lyap12}.

{\bf (iii)} Let $\nu$ be any measure on $\T\cong \R/\Z$, such that the set of rationals has zero measure (in particular, any continuous measure). 
If $\chi(\Mc_{_\Om})<\half \lam$, then for $\P$-a.e.\ $\ba$ the maximal spectral type of
$(X_\ba,T,\mu_\ba)$ is singular with respect to $\nu$. In particular, a.e.\ such system is spectrally disjoint from any given weakly mixing measure-preserving transformation.
\end{theorem}

Here (iii) is the main claim of the theorem; (i) and (ii) are essentially a summary of what we have already shown.

\begin{proof} (i) Properties {\bf (C1)-(C2)} are immediate from the assumptions, {\bf (C3)} follows from {\bf (B3)}, since $\P$ is a Bernoulli 
measure, and {\bf (C4)} follows from {\bf (B1)}. The remaining part of the claim was already noted in Section~\ref{sec:rand}.

\smallskip

(ii)
 Applying Theorem~\ref{th-equi1} and Proposition~\ref{prop-equi1}, we obtain that, given any $\bt \in \T^d\setminus \Q^d/\Z^d$, for $\P$-a.e.\ $\ba$
the orbit $\Ga^n(\ba,\bt)$ is equidistributed with respect to the measure $\P\times m_d$. 
Note also that conditions {\bf (A), (B)}
are preserved when passing from $\mu$ to $\mu^{*k}$. Thus we have

\begin{lemma} \label{lem:equi2} Under the assumptions of the theorem,
given any $\bt \in \T^d\setminus \Q^d/\Z^d$, for $\P$-a.e.\ $\ba$, for all $k\in \N$, the orbit
$$
\{\Ga^{kn}(\ba,\bt)\}_{n\in \N}
$$
is equidistributed with respect to the measure $\P\times m_d$.
\end{lemma}

This equidistribution result implies that the skew product $\Ga$ is ergodic,
hence the Lyapunov exponent 
$\chi(\Cc_\Om)$ of the spectral cocycle is well-defined.

\smallskip

(iii) The main part of the proof is contained in the following

\begin{lemma} \label{lem:estLyap}
Suppose that all the assumptions of the theorem are satisfied. Let $\om\vec s\not \in \Q^d$. Then for $\P$-a.e.\ $\ba\in \Om$,
$$
\chi^+_{\ba,\om\vec s} \le \chi(\Mc_{_\Om}) \le \textstyle{\frac{1}{2}}\lam.
$$
\end{lemma}

\begin{proof}[Proof of the lemma] 
This is an extension of an argument from \cite{BuSo22} (more precisely, of the first part of the proof of \cite[Theorem 2.4]{BuSo22}).
Denote
$$
\Cc^{[n]}(\ba,\bt) :=\Cc_{_\Om}((\ba,\bt),n)= \Cc(\Ga^{n-1}(\ba,\bt))\cdot \ldots \cdot \Cc(\ba,\bt)
$$
for convenience. Observe that for any $\ba\in \Om, \bt\in \T^d$, and $k\in \N$ we have
\begin{equation} \label{est1}
\chi_{\ba,\bt}^+ = \limsup_{n\to \infty} \frac{1}{nk} \log\|\Cc^{[nk]}(\ba,\bt)\|.
\end{equation}
Indeed, the inequality $\ge$ is obvious, and the reverse inequality follows from 
\begin{eqnarray*}
\|\Cc^{[nk+j]}(\ba,\bt)\| & \le & \|\Cc^{[nk]}(\ba,\bt)\|\cdot \|\Cc^{[j]}(\Ga^{nk}(\ba,\bt))\| \\
& \le & \|\Cc^{[nk]}(\ba,\bt)\|\cdot \max_{i\le \ell}\|\Sf_i\|^{k-1},
\end{eqnarray*}
for $1\le j \le k-1$.

Now, by the definition of the cocycle, in view of (\ref{est1}), for any $\ba\in \Om$ and $k\in \N$,
\begin{equation}\label{est2}
\chi_{\ba,\om\vec s}^+ \le \limsup_{n\to \infty} \frac{1}{nk} \sum_{j=0}^{n-1} \log\|\Cc^{[k]}(\Ga^{kj}(\ba,\om\vec s))\|.
\end{equation}
Fix any $k\in \N$ and $\om\vec s\not\in \Q^d$. 
By \eqref{est2} and Lemma~\ref{lem:equi2}, for $\P$-a.e.\ $\ba$,
\begin{eqnarray*}
\chi^+_{\ba,\om \vec s} & \le & \lim_{\eps\to 0} \limsup_{n\to \infty} \frac{1}{nk} \sum_{j=0}^{n-1} \log\bigl(\eps+ \|\Cc^{[k]}(\Ga^{kj}(\ba,\bt))\|\bigr) \\
& = & \lim_{\eps\to 0} \frac{1}{k} \int_{\Om}\int_{\T^d} \log\bigl(\eps + \|\Cc^{[k]}(\bx,\bt)\|\bigr)\,dm_d(\bt)\,d\P(\bx),
\end{eqnarray*}
using the definition of equidistribution and the fact that the integrand is continuous. 

We claim that  we can pass to the limit $\eps\to 0$ by dominated convergence to obtain for all $k\in \N$:
$$
\chi^+_{\ba,\om \vec s} \le \frac{1}{k} \int_{\Om}\int_{\T^d} \log\|\Cc^{[k]}(\bx,\bt)\|\,d\bt\,d\P(\bx),\ \ \mbox{for $\P$-a.e.\ $\ba$},
$$
hence 
\begin{equation} \label{speca1}
\chi^+_{\ba,\om \vec s} \le \chi(\Cc_{_\Om}) \le \frac{\lam}{2}, \ \ \mbox{for $\P$-a.e.\ $\ba$},
\end{equation}
as desired.
In order to justify passing to the limit $\eps\to 0$, we split the integral into two:
$$
\int_{\Om}\int_{\T^d} \log\bigl(\eps + \|\Cc^{[k]}(\bx,\bt)\|\bigr)\,dm_d(\bt)\,d\P(\bx) =: \int_{A_1} + \int_{A_2},
$$
where 
$$
A_1 = \bigl\{(\bx,\bt) \in \Om\times \T^d:\ \|\Cc^{[k]}(\bx,\bt)\| < 1/2\bigr\},\ \ \mbox{and} \  \ \ A_2 = (\Om\times \T^d)\setminus A_1.
$$
Without loss of generality we can assume that $\eps \in (0,\half)$. Then for $(\bx,\bt)\in A_1$,
$$
\Bigl|\log\bigl| \eps + \|\Cc^{[k]}(\bx,\bt)\|\bigr|\Bigr| = - \log\bigl(\eps + \|\Cc^{[k]}(\bx,\bt)\|\bigr) < - \log \|\Cc^{[k]}(\bx,\bt)\|,
$$
which is positive and integrable, so the Dominated Convergence Theorem applies. On the other hand, for $(\bx,\bt)\in A_2$ we have
$$
\Bigl|\log\bigl| \eps + \|\Cc^{[k]}(\bx,\bt)\|\bigr| \Bigr|\le \max\bigl\{\log 2, \log(1/2 + C_1^k)\bigr\},
$$
where $C_1 = \max_{j\le \ell} \|\Sf_j\|$, so again passing to the limit $\eps\to 0$ is justified.
\end{proof}

Now we can conclude the proof of the theorem.
In view of \eqref{eq-spec},
it suffices to prove the analogous claim for the suspension flow $(\Xx_\ba^{\vec s}, h_\tau,\wt\mu_\ba)$, with the vector of heights $\vec s =\one$, and
assume that $\nu$ is a measure on $\R$. Moreover, by the discussion in Section~\ref{sec:spec}, it is enough 
to show that
for $\P$-a.e.\ $\ba$ we have $\sig_f\perp \nu$ for any Lip-cylindrical function $f$ (of any level).
 By Lemmas~\ref{lem:equi2} and \ref{lem:estLyap}, we have that for any irrational $\om\in \R$, hence for
$\nu$-a.e. $\om$, for $\P$-a.e.\ $\ba$,
\be \label{last1}
\chi^+_{\ba,\om\one} \le \chi(\Mc_{_\Om}) < \textstyle{\frac{1}{2}}\lam.
\ee
Then, by Fubini's Theorem, for $\P$-a.e.\ $\ba$ the inequality holds for $\nu$-a.e. $\om$. Fix such an $\ba\in \Om$ and consider a cylindrical function $f$ for the 
suspension flow $(\Xx_\ba^{\one}, h_\tau,\wt\mu_\ba)$. By Corollary~\ref{cor-BuSo}, $\chi^+_{\ba,\om\one}\ge \half\lam$ for $\sig_f$-a.e.\ $\om$. Hence $\sig_f$ and $\nu$ are mutually singular.
\end{proof}

\begin{corollary} \label{cor:ex}
Let $\zeta_1,\ldots,\zeta_\ell$, $\ell\ge 2$, be a family of primitive substitutions such that their substitution matrices $\Sf_j\in SL(d,\R)$ generate a semigroup which
acts strongly irreducibly on $\R^d$. Consider a random $S$-adic system corresponding to a directive sequence $\ba\in \Om = \{\zeta_1,\ldots,\zeta_\ell\}^\N$, 
equipped with a Bernoulli measure
$\P=(p_1,\ldots,p_\ell)^\N$. Suppose that for $\P$-a.e.\ $\ba\in \Om$ the space $X_\ba$ is aperiodic.
Let $\Aa(\ba) = \Sf^{\sf T}_{\ba_1}$ be the $SL(d,\R)$ cocycle over $(\Om,\sig,\P)$, defined in \eqref{Rauzy-Veech1}, and let $\lam>0$ be its top
Lyapunov exponent. If 
\begin{equation} \label{Lyap3}
\int_{\T^d} \log\|\Cc_{\zeta_j}(\bt)\|\,dm_d(\bt) < \textstyle{\frac{1}{2}}\lam,\ \ \mbox{for all} \ j\le \ell,
\end{equation}
then for $\P$-a.e.\ $\ba\in \Om$ the $S$-adic system $(X_\ba, T, \mu_\ba)$ has singular spectrum. Moreover, for any given weakly mixing measure-preserving system, 
for $\P$-a.e.\ $\ba\in \Om$ 
the $S$-adic system $(X_\ba, T, \mu_\ba)$ is spectrally disjoint from it.
\end{corollary}

\begin{proof}
By the definition of the measure $\P$, equation \eqref{Lyap12} implies
$$
\chi(\Cc_{_\Om}) \le  \sum_{j=1}^\ell p_j \int_{\T^d} \log\|\Cc_{\zeta_j}(\bt)\|\,dm_d(\bt) < \textstyle{\frac{1}{2}}\lam,
$$
in view of our assumptions,
and the claim follows from Theorem~\ref{th:main}.
\end{proof}

\section{Examples}

It is not so easy to find examples illustrating Theorem~\ref{th:main}, without relying on numerical methods. There are difficulties both in estimating the Lyapunov exponent of the spectral cocycle from above and the Lyapunov exponent $\lam$ of the ``non-twisted'' cocycle, associated with the random $S$-adic system, from below.
Here we provide a class of examples, which demonstrates
that the theorem is non-vacuous.

\begin{example}
Consider the family of substitutions
$$
\zeta_m:\ 0\mapsto 0^{2m}1^{m^2}2,\ 1\mapsto 0,\ 2\mapsto 1,
$$
for $m\in \N$, and let $X_\ba$ be the $S$-adic system associated with the directive sequence $\ba\in \Om=\{\zeta_m,\zeta_{m+1}\}^\N$. Let $\P = (p_1,p_2)^\N$ be any Bernoulli 
measure on $\Om$. We claim that for $m\ge 23$ this system satisfies the conditions of Corollary~\ref{cor:ex}.
\end{example}

\begin{proof} Note that all  the assumptions of Theorem~\ref{th:main} are satisfied. Strong irreducibility follows by an examination of the eigenvectors
of the matrices. Aperiodicity follows from Proposition~\ref{prop-aper}, since the composition of every two substitutions $\zeta_m,\zeta_n$
is left proper.

\smallskip

 {\sc Step 1.}
First we obtain the lower bound for the Lyapunov exponent $\lam$. Since the norm of a matrix is equal to the norm of the transpose, it suffices to estimate from below the norm of
a product
$
M_{i_1}\cdots M_{i_n},
$
where 
$$
M_{i_j}\in \{M_1, M_2\} := \{\Sf_{\zeta_m}, \Sf_{\zeta_{m+1}}\} = \left\{ \left(\begin{array}{ccc} 2m & 1 & 0 \\ m^2 & 0 & 1 \\ 1 & 0 & 0 \end{array}\right),\ \ 
\left(\begin{array}{ccc} 2m+2 & 1 & 0 \\ (m+1)^2 & 0 & 1 \\ 1 & 0 & 0 \end{array}\right) \right\}
$$
Let $\Gam$ be the semigroup generated by $\{M_1,M_2\}\subset SL(3,\R)$. 

\medskip

{\sc Claim 1.} {\em The semigroup $\Gam$ preserves the following cone:
\begin{equation} \label{cone}
\Crs:= \{\bf 0\} \cup \left\{ \left( \begin{array}{c} x_1 \\ x_2 \\ x_3 \end{array} \right):\ \ \begin{array}{l} 2.3m \le x_1/x_3 \le 2.5m+3 \\ m^2 \le x_2/x_3 \le m^2+2m+2\\ x_3>0 \end{array} \right\}
\end{equation}
}
For the verification, see the next section.

\medskip

{\sc Claim 2.} {\em Each of the matrices $M_1,M_2$ expands the 1-norm of any nonzero vector $\bx \in \Crs$ at least by a factor of $1.9m$, hence the Lyapunov exponent satisfies
$$
\lam\ge \log(1.9m).
$$
}
For the verification, see the next section.

\medskip

\noindent {\bf Remark.} It is similarly possible to estimate the Lyapunov exponent from above to obtain that $\lam\le \log(3m)$ for $m\ge 16$. These estimates can certainly be tightened
for larger $m$.

\medskip

{\sc Step 2.} Next we estimate from above the Lyapunov exponent of the spectral cocycle, starting with the formula for the matrix-function that defines it:
$$
\Cc_{\zeta_m}(z_0,z_1,z_2) = \left( \begin{array}{ccc} 1 + z_0 + \cdots + z_0^{2m-1} & z_0^{2m}( 1+z_1 + \cdots + z_1^{m^2-1}) & z_0^{2m} z_1^{m^2} \\
                                                                1 & 0 & 0 \\ 0 & 1 &  0 \end{array} \right),\ \ \ z_j = e^{-2\pi i t_j}.
$$
The next estimate follows the template of \cite{BGM} and \cite{BuSo22}.
In \eqref{Lyap3} we can use the Frobenius norm of the matrix-function to obtain
\begin{eqnarray*}
\log{\|\Cc_{\zeta_m}(\bz)\|}^2_{\rm F} & = & \log\left(3 + \left|\frac{z_0^{2m}-1}{z_0-1}\right|^2 + \left|\frac{z_1^{m^2}-1}{z_1-1}\right|^2\right) \\
                                                   & = & \log\left( 3|z_0-1|^2 |z_1-1|^2 + |z_0^{2m}-1|^2 |z_1-1|^2 + |z_0-1|^2 |z_1^{m^2}-1|^2\right) \\
                                                   &    & -2\log|z_0-1| - 2\log|z_1-1| \\
                                                   & \le & \log\left(12|z_0-1|^2 + 16 + 4  |z_0-1|^2\right)- 2\log|z_0-1| - 2\log|z_1-1|,
\end{eqnarray*}
estimating $|z_1-1|\le 2$, etc. Now we split the integral in \eqref{Lyap3} into the sum of three integrals, each of which reduces to one-dimensional. Further, we apply the formula for the
logarithmic Mahler measure of a polynomial. For a polynomial of a single variable $p(z)$ it is given by
$$
\mfr{m}(p)= \log M(p),\ \ \ M(p) = \exp \int_0^1 \log|p(e^{2\pi i t})|\,dt = |a_0|\cdot \prod_{j\ge 1} \max\{|\alpha_j|,1\},
$$
where $a_0$ is the leading coefficient and $\alpha_j$ are the complex zeros of $p$. Thus,
$$
\int_{\T^3} \log |z_0-1|\,dm_3(\bz) = \int_{\T^3} \log |z_1-1|\,dm_3(\bz) = \mfr{m}(z-1)=0.
$$
Further, writing $z=z_0$ we have $16(|z_0-1|^2 +1) = 16|2-z-z^{-1} +1| = 16|z^2 - 3z +1|$, hence
\begin{eqnarray*}
\int_{\T^3} \log{\|\Cc_{\zeta_m}(\bz)\|}^2_{\rm F}\,dm_3(\bz) & \le &  \log 16 + \mfr{m}(z^2-3z+1) \\
& = & 4\log 2 + \log[(3+\sqrt{5})/2] \\[1.1ex] & = & \log[8(3+\sqrt{5})].
\end{eqnarray*}
Thus, by Claim 2, $\int_{_{\T^3}} \log{\|\Cc_{\zeta_m}(\bz)\|}^2_{\rm F}\,dm_3(\bz)< \half \lam$ is guaranteed whenever $\half \log[8(3+\sqrt{5})] < \half \log(1.9m)$, and this holds for
$m\ge 23$, as desired. Observe that our estimate in the end did not depend on $m$, so we obtain the same upper bound for 
$\int_{_{\T^3}} \log{\|\Cc_{\zeta_{m+1}}(\bz)\|}^2_{\rm F}\,dm_3(\bz)$, completing the proof.
\end{proof}

\begin{remark} {\em 1. One can introduce more variety by considering other substitutions with the same substitution matrices $M_1, M_2$, for instance
$$
\zeta_{m,k}: \ 0\mapsto 0^k 2 0^{2m-k} 1^{m^2},\ 1\mapsto 0,\ 2\mapsto 1,\ \ 1 \le k\le 2m,
$$
and analogously with $\zeta_{m+1,k}$.
One can choose any number of them with arbitrary probabilities. An argument similar to Step 2 yields that the random $S$-adic system satisfies the conditions of
Corollary~\ref{cor:ex} for $m\ge 26$. We leave the details to the reader.

2. The substitution $\zeta_m$ was chosen because it is easy to estimate the eigenvalues of its substitution matrix for $m$ large: they are $\theta_1 \approx (1+\sqrt{2})m,\
\theta_2 \approx (1-\sqrt{2})m, \ \theta_3 \approx m^{-2}$. Thus, this is a non-Pisot matrix. Moreover, one has that the second Lyapunov exponent of the cocycle $A$ from
\eqref{Rauzy-Veech1} is positive: see the appendix for more details.
}
\end{remark}

\section{Details of the computation.}
\begin{proof}[Proof of Claim 1.]
Let $\bx = (x_1, x_2, x_3)^{\sf T}\in \Crs$. Then
$$
x_1^{-1} M_1 \bx = \left(\begin{array}{c}  x_2/x_1 + 2m \\ x_3/x_1 + m^2 \\ 1 \end{array} \right),\ \ \ \ 
x_1^{-1} M_2 \bx = \left(\begin{array}{c}  x_2/x_1 + 2m+2 \\ x_3/x_1 + (m+1)^2 \\ 1 \end{array} \right).
$$
It is enough to verify the inequalities
\begin{equation}\label{ineq1}
2.3 m \le x_2/x_1 + 2m < x_2/x_1 + 2m+2 \le 2.5m+3
\end{equation}
and
\begin{equation}\label{ineq2}
m^2  \le x_3/x_1 + m^2 < x_3/x_1 + (m+1)^2  \le m^2+2m+2.
\end{equation}
By the definition of the cone $\Crs$, see \eqref{cone},
$$
\frac{m^2}{2.5m+3} \le \frac{x_2}{x_1} \le \frac{m^2+2m+2}{2.3m}\,.
$$
Thus, the left inequality in \eqref{ineq1} follows from $\frac{m^2}{2.5m+3} \ge 0.3m$, and the right inequality in \eqref{ineq1} follows from
$\frac{m^2+2m+2}{2.3m}\le 0.5m+1$, both of which hold for $m\ge 4$. The left inequality in \eqref{ineq2} is obvious, whereas the right inequality holds for all $m\ge 1$, since
$x_3/x_1 \le (2.3m)^{-1}$.
\end{proof}

\begin{proof}[Proof of Claim 2.] If $\bx\in \Crs\setminus \{\mathbf{0}\}$, then
$$
{\|\bx\|}_1 = x_1 + x_2 + x_3 \le x_1 + (m^2+2m+3)x_3 \le \Bigl(1 + \frac{m^2+2m+2}{2.3m}\Bigr)x_1,
$$
whereas
$$
{M_1\|\bx\|}_1 = (m+1)^2x_1 + x_2 + x_3 \ge (m+1)^2 x_1 + (m^2+1)x_3 \ge \Bigl((m+1)^2 + \frac{m^2+1}{2.5m+3}\Bigr)x_1,
$$
hence
$$
\frac{{M_1\|\bx\|}_1}{{\|\bx\|}_1} \ge \frac{(m+1)^2 + \frac{m^2+1}{2.5m+3}}{1 + \frac{m^2+2m+2}{2.3m}} \ge 1.9m,\ \ \ \mbox{for}\ m\ge 8,
$$
which is a straightforward, but slightly tedious computation, simplified by the observation that $\frac{m^2+1}{2.5m+3}\ge m/3$ for $m\ge 5$.
\end{proof}

\newpage

\appendix

\section{Lyapunov exponents and weak-mixing for $S$-adic systems}
\smallskip
\begin{center}by \textsc{Pascal Hubert and Carlos Matheus}\end{center}
\medskip

\subsection{Weak-mixing for certain $S$-adic systems}
In this section, we explain how Avila--Forni's famous result \cite{AF} about weak-mixing for interval exchange transformations can be extended to certain random $S$-adic systems.

\begin{theorem}\label{t.A} Let $\Omega$ be a mixing shift of finite type  on a finite alphabet $\B$,  $\varphi: \Omega \to \Rset$ a H\"older potential and $\mu$ the Gibbs probability measure associated to this potential. 
For each element $b \in \B$, we consider a substitution $\zeta_b$ acting on $\A^\Nset$ where $\A$ is a finite alphabet with $d$ letters. 
The matrices of  the substitution induce a cocycle $S$ over $\Omega$ with values in the set of $d\times d$ matrices with entries in $\mathbb{N}$. We suppose  that the following properties hold:  

\noindent {\bf (C1)} For each $b \in \B$, $S_{b}$ has determinant 1.

\noindent {\bf (C2)} There is a word $q$ in the language of $\Omega$ such that all the entries of the matrix $\textrm{S}_q$ are positive
and the substitution $\zeta_q$ satisfies the strong coincidence condition.

\noindent {\bf (C3)} The cocycles $\textrm{S}$ and $\textrm{S}^{-1}$ are log-integrable with respect to $\mu$.

\noindent {\bf (C4)} The second exponent of the cocycle $\textrm{S}$ is positive.


\noindent {\bf (C5)} The group generated by  $\textrm{S}_{q}$, $q$ in the language of $\Omega$, is Zariski dense in  $\textrm{SL}(d,\Rset)$.

Then, for $\mu$ almost every sequence $\underline{a} \in \A^\Nset$, the associated S-adic system is weak-mixing.
\end{theorem}

\begin{proof} We only give a sketch of proof by discussing how to use the arguments in Avila--Forni \cite{AF}. 
Due to (C2), for almost every parameter with respect to $\mu$ the S-adic system is uniquely ergodic.

Section 2 in Avila--Forni \cite{AF} concerns uniform integral locally constant cocycles over strongly expanding systems on simplices. The strongly expanding property is also called distortion estimate in some texts (see, e.g., Avila--Leguil \cite{AL}). In our setting, this is obtained by inducing on a cylinder $q$ where the matrix has all its entries positive (in particular, cylinders replace simplices in Avila--Forni's discussions). In general, the induced shift is then defined on a countable alphabet, but  since we start from a subshift $\Omega$ of finite type, the so-called BIP (big images and pre-images) property is automatically satisfied: see, e.g., Definition 3.5 in Sarig \cite{Sa}. Thus, one can apply Theorem 5.9 from Sarig's survey \cite{Sa} to get exponential decay of correlations for Lipschitz observables. Moreover, from the techniques explained in Sarig's survey \cite{Sa} in connection to Gibbs property of $\mu$, the desired distortion estimate holds.


Section 3 in Avila--Forni \cite{AF} contains an abstract result that applies here. Furthermore, the setting of Section 4 in Avila--Forni \cite{AF} is the same for us thanks to our assumption (C4) that the second Lyapunov exponent of the cocycle $\textrm{S}$ is positive.

Theorem 5.1 in Avila--Forni \cite{AF} is valid in our case because Zariski density assumption in (C5) implies that the cocycle is twisting (in the sense of Avila--Viana \cite{AV}) which is exactly how the proof works. In fact, the statement of \cite[Thm. 5.1]{AF} says that a given affine line $L$ does not intersect the  central stable manifold of a generic parameter.
This is true if the 2-dimensional subspace containing the line $L$ and the central stable manifold are transversal.
The Markov property and Zariski density imply that one can connect arbitrary words  by cylinders whose matrices send the direction of $L$ to independent vectors. This is in contradiction with the fact that the codimension of the central stable space is at least two\footnote{More precisely,
there is a folklore result, as follows: Given a Zariski dense monoid $\Mk$ of, say, $SL(d,\R)$ matrices, one can find a finite subset $F\subset \Mk$ with the
property that for any $k<d$ and any pair of subspaces $E_1, E_2$ of $\R^d$ of dimensions $k$ and $d-k$ respectively, there exists $A\in F$ such that
$A(E_1) \cap E_2 = \{0\}$. See, e.g., \cite[Lemma 4.3]{BreGe} for $k=1$; the case of $k>1$ follows by considering the exterior power; see e.g., the proof of
\cite[Theorem 4.2]{FHM}. Note also that Zariski density of the group is equivalent to the Zariski density of the semigroup.}. 

Finally, Theorem 6.1 in \cite{AF} (originally due to Veech) is known to be valid for $S$-adic systems, under the assumption that the substitution
$\zeta_q$ for a word $q$, which appears in the directive sequence infinitely often, satisfies the strong coincidence condition, see Arbul\'u and Durand
\cite[Proposition 5.1]{AD}.
 This holds $\mu$-almost surely under our argument, and the argument is complete\footnote{
 There are other, ``Veech-type'' conditions in the literature, see,
 e.g., \cite{BCY22,Mercat}.}.
\end{proof}

\subsection{Some examples}

In this section, we show that the assumptions of Theorem \ref{t.A} above applies to the families of substitutions considered by Solomyak in Example 4.1 above. In particular, we solve the question in Remark 4.2 above concerning the positivity of the second Lyapunov exponent attached to these examples. 
Note that the composition of every two substitutions $\zeta_m,\zeta_n$ is left proper, hence the strong coincidence assumption holds.

\subsubsection{Invariant cones}

Let $m\geq 3$ be an integer, and recall that the matrices appearing in Example 4.1 above are:   
$$A=\left(\begin{array}{ccc} 2m & 1 & 0 \\ m^2 & 0 & 1 \\ 1 & 0 & 0 \end{array}\right) \quad \textrm{and} \quad B=\left(\begin{array}{ccc} 2(m+1) & 1 & 0 \\ (m+1)^2 & 0 & 1 \\ 1 & 0 & 0 \end{array}\right).$$ 

Note that the inverses of these matrices are 
$$A^{-1}=\left(\begin{array}{ccc} 0 & 0 & 1 \\ 1 & 0 & -2m \\ 0 & 1 & -m^2 \end{array}\right) \quad \textrm{and} \quad B^{-1}=\left(\begin{array}{ccc} 0 & 0 & 1 \\ 1 & 0 & -2(m+1) \\ 0 & 1 & -(m+1)^2 \end{array}\right).$$ 

Define the cone 
$$\mathcal{C} = \{(0,0,0)\}\cup \left\{(x_1,x_2,x_3)\in\mathbb{R}^3: x_1\neq 0, -3m\leq \frac{x_2}{x_1}\leq -2m, -m^2-3m\leq \frac{x_3}{x_1}\leq -m^2+1\right\}.$$

\begin{prop}\label{p.1} The cone $\mathcal{C}$ is invariant under the semigroup generated by $A^{-1}$ and $B^{-1}$. 
\end{prop}

\begin{proof} By definition, 
$$\left(\begin{array}{c} x_1' \\ x_2' \\ x_3' \end{array}\right):=A^{-1}\left(\begin{array}{c} x_1 \\ x_2 \\ x_3\end{array}\right) = \left(\begin{array}{c} x_3 \\ x_1-2mx_3 \\ x_2-m^2 x_3 \end{array}\right).$$ 
Therefore, if $\vec{x}=(x_1,x_2,x_3)\in\mathcal{C}\setminus\{(0,0,0)\}$, then 
$$\frac{x_2'}{x_1'} = \frac{x_1}{x_3}-2m\in \left[-\frac{1}{m^2-1}-2m,-\frac{1}{m^2+3m}-2m\right]$$ 
and 
$$\frac{x_3'}{x_1'} = \frac{x_2}{x_3}-m^2 = \frac{(-x_2/x_1)}{(-x_3/x_1)}-m^2\in\left[\frac{2m}{m^2+3m}-m^2,\frac{3m}{m^2-1}-m^2\right],$$
so that $A^{-1}(\vec{x})$ falls in the interior of $\mathcal{C}$ when $m\geq 3$.

Similarly, 
$$\left(\begin{array}{c} x_1'' \\ x_2'' \\ x_3'' \end{array}\right):=B^{-1}\left(\begin{array}{c} x_1 \\ x_2 \\ x_3\end{array}\right) = \left(\begin{array}{c} x_3 \\ x_1-2(m+1)x_3 \\ x_2-(m+1)^2 x_3 \end{array}\right).$$ 
Hence, if $\vec{x}=(x_1,x_2,x_3)\in\mathcal{C}\setminus\{(0,0,0)\}$, then 
$$\frac{x_2''}{x_1''} = \frac{x_1}{x_3}-2-2m\in \left[-\frac{1}{m^2-1}-2-2m,-\frac{1}{m^2+3m}-2-2m\right]$$ 
and 
$$\frac{x_3''}{x_1''} = \frac{x_2}{x_3}-(m+1)^2 = \frac{(-x_2/x_1)}{(-x_3/x_1)}-1-2m-m^2\in\left[\frac{2m}{m^2+3m}-1-2m-m^2,\frac{3m}{m^2-1}-1-2m-m^2\right],$$
so that $B^{-1}(\vec{x})$ falls in the interior of $\mathcal{C}$ when $m\geq 3$.
 \end{proof}

\subsubsection{Uniform expansion} Once an adequate invariant cone was identified, let us show that the vectors in this cone are expanded: 

\begin{prop}\label{p.2} For each $\vec{x}\in\mathcal{C}\setminus\{(0,0,0)\}$, one has 
$$\frac{\|A^{-1}(\vec{x})\|_{L^1}}{\|\vec{x}\|_{L^1}}\geq \frac{m^4+2m^3-5m}{m^2+6m+1} \quad \textrm{and} \quad \frac{\|B^{-1}(\vec{x})\|_{L^1}}{\|\vec{x}\|_{L^1}}\geq \frac{m^4+4m^3+3m^2-7m-3}{m^2+6m+1}$$
\end{prop}

\begin{proof} We can assume without loss of generality that $x_1>0$. In this case, the fact that $\vec{x}\in\mathcal{C}$ implies that $x_3 < 0$ and $x_2-m^2x_3, x_2-(m+1)^2x_3>0$, so that  
$$\|\vec{x}\|_{L^1} = x_1-x_2-x_3 = x_1(1-x_2/x_1-x_3/x_1)\leq x_1(1+6m+m^2)$$ 
and 
\begin{eqnarray*}
\|A^{-1}(\vec{x})\|_{L^1} &=& -x_3 + x_1-2mx_3 + x_2-m^2x_3 = x_1\left(1+\frac{x_2}{x_1}-(m+1)^2\frac{x_3}{x_1}\right) \\ 
&\geq& x_1(1-3m+(m+1)^2(m^2-1)), 
\end{eqnarray*}
\begin{eqnarray*}
\|B^{-1}(\vec{x})\|_{L^1} &=& -x_3 + x_1-2(m+1)x_3 + x_2-(m+1)^2x_3 = x_1\left(1+\frac{x_2}{x_1}-(m+2)^2\frac{x_3}{x_1}\right) \\ 
&\geq& x_1(1-3m+(m+2)^2(m^2-1)).
\end{eqnarray*} 
This completes the argument. 
\end{proof} 

\subsubsection{Conclusions}

The second Lyapunov exponent $\lambda_2$ of the semigroup generated by $A$ and $B$ is positive. In fact, Solomyak showed that the top Lyapunov exponent $\lambda_1$ of $A$ and $B$ satisfies $\log(3m)\geq \lambda_1\geq \log(1.9 m)$ for $m\geq 23$. On the other hand, Propositions \ref{p.1} and \ref{p.2} imply that the top Lyapunov exponent $\mu_1$ of the inverse cocycle (generated by $A^{-1}$ and $B^{-1}$) is $\geq \log(0.9 m^2)$ for $m\geq 35$. Since the third Lyapunov exponent $\lambda_3$ of the cocycle generated by $A, B\in SL(3,\mathbb{Z})$ satisfies $\lambda_3=-\mu_1$ and $\lambda_1+\lambda_2+\lambda_3=0$, we obtain that $\lambda_2\geq \log(0.9 m^2) - \log(3m)>0$ for $m\geq 35$. 

Also, it is not hard to check that $A$ and $B$ generate a Zariski dense subgroup of $SL(3,\mathbb{R})$. Indeed, the characteristic polynomial of $A$ is $1+m^2 x +2m x^2 - x^3$ has discriminant $8m^6-68m^3-27$. Therefore, the three eigenvalues of $A$ (and $B$) are real. Since $8m^6-68m^3-27$ is an irreducible polynomial of $m$, Siegel's theorem says that its values are not a square for all but finitely many choices of $m$. Hence, the Galois groups of the characteristic polynomials of $A$ and $B$ are the full symmetric group $\textrm{Sym}_3$ provided these polynomials are irreducible and, as it turns out, this is the case when $m=3, 4$ modulo $17$. Also, the roots of the characteristic polynomials of $A$ and $B$ have distinct moduli (cf. Remark 4.2 above), so that these matrices are Galois-pinching. Finally, we have that $A$ and $B$ have infinite order and they do not commute. From these facts, we can derive the Zariski density of the semi-group generated by $A$ and $B$ by applying the results of Prasad--Rapinchuk \cite{PR}. 

Consequently, the family of substitutions introduced by Solomyak in Example 4.1 above satisfy the assumptions of Theorem \ref{t.A} whenever $m=3$ modulo $17$ is sufficiently large. 

\medskip

{\bf Acknowledgement.} 
We are grateful to Felipe Arbul\'u for the clarification of the aperiodicity issue and for pointing out that an assumption was missing for the validity of the Veech criterion for $S$-adic systems in our original version.

\end{document}